\documentclass[preprint]{elsarticle}

\usepackage{lineno,hyperref}
\modulolinenumbers[5]
\usepackage[utf8]{inputenc}
\usepackage{subfigure}

\usepackage{amsmath,amsfonts,amssymb,amsthm}
\usepackage{enumerate}
\usepackage{url}
\usepackage{textcomp}
\usepackage{array}
\usepackage{tikz}
\usetikzlibrary{decorations.pathreplacing,calc}
\usetikzlibrary{decorations.markings}
\usetikzlibrary{decorations.pathmorphing}

\usepackage{todonotes}

\newcommand{\N}{\mathbb{N}}
\newcommand{\Z}{\mathbb{Z}}
\renewcommand{\L}{\mathcal{L}}

\newcommand{\size}[1]{|#1|}
\newcommand{\leafset}[1]{#1_{\mbox{\scriptsize\textleaf}}}
\newcommand{\leaf}[1]{|#1|_{\mbox{\scriptsize\textleaf}}}
\newcommand{\Alp}{\mathrm{Alph}}
\newcommand{\pref}{\mathrm{pref}}
\newcommand{\suff}{\mathrm{suff}}
\newcommand{\Pref}{\mathrm{Pref}}
\newcommand{\Suff}{\mathrm{Suff}}
\newcommand{\Fact}{\mathrm{Fact}}

\newcommand{\bigo}{\mathcal{O}}

\newcommand{\Graphs}{\mathcal{G}}
\newcommand{\Trees}{\mathcal{T}}
\newcommand{\Cater}{\mathcal{C}}

\newcommand{\infl}{\omega}
\newcommand{\Left}{\mathrm{Left}}
\newcommand{\Right}{\mathrm{Right}}
\newcommand{\RC}{\mathrm{RC}}
\newcommand{\PNW}{\mathcal{PNW}}
\newcommand{\kPNW}{k\mbox{-}\PNW}
\DeclareMathOperator{\spine}{Deg}
\newcommand{\Sub}{\mathcal{S}}

\newtheorem{theorem}{Theorem}[section]

\newtheorem{lemma}{Lemma}[section]
\newtheorem{proposition}{Proposition}[section]
\newtheorem{corollary}{Corollary}[section]
\theoremstyle{definition}
\newtheorem{definition}{Definition}[section]
\newtheorem{example}{Example}[section]

\newtheorem{remark}{Remark}[section]
\newtheorem{problem}{Problem}[section]

\journal{Theoretical Computer Science}

\begin{document}

\begin{frontmatter}

\title{Leaf realization problem, caterpillar graphs and prefix normal words}

\author[lacim]{Alexandre Blondin Mass\'e \fnref{subnserc} \corref{mycorrespondingauthor}}
\author[uqtr]{Julien de Carufel}
\author[uqtr]{Alain Goupil}
\author[lacim]{M\'elodie Lapointe \fnref{scholarship}}
\author[lacim]{\'Emile Nadeau  \fnref{scholarship}}
\author[lacim]{\'Elise Vandomme}
\fntext[subnserc]{Supported by a grant from the National Sciences and Engineering Research Council of Canada (NSERC) through Individual Discovery Grant RGPIN-417269-2013}
\fntext[scholarship]{Supported by a scholarship from the NSERC}

\address[lacim]{Laboratoire de Combinatoire et d'Informatique Math\'ematique, Universit\'e du Qu\'ebec \`a Montr\'eal, Canada}
\address[uqtr]{Laboratoire Interdisciplinaire de Recherche en Imagerie et en Combinatoire, Universit\'e du Qu\'ebec \`a Trois-Rivi\`eres, Canada}

\begin{abstract}
    Given a simple graph $G$ with $n$ vertices and a natural number  $i\leq n$, let $L_G(i)$ be the maximum number of leaves that can be realized by an induced subtree $T$ of $G$ with  $i$ vertices.
    We introduce a problem that we call the \emph{leaf realization problem}, which consists in deciding whether, for a given sequence of $n + 1$ natural numbers $(\ell_0,\ell_1,\ldots,\ell_n)$, there exists a simple graph $G$ with $n$ vertices such that $\ell_i = L_G(i)$ for $i = 0,1,\ldots,n$.
    We present basic observations on the structure of these sequences for general graphs and trees.
    In the particular case where $G$ is a caterpillar graph, we exhibit a bijection between the set of the discrete derivatives of the form $(\Delta L_G(i))_{1 \leq i \leq n - 3}$ and the set of prefix normal words.
\end{abstract}

\begin{keyword}
    graph theory \sep combinatorics on words \sep induced subtrees \sep leaf
    \sep prefix normal words \sep prefix normal form
\end{keyword}

\end{frontmatter}


\section{Introduction}

Recently, a family of binary words, called \emph{prefix normal words}, was introduced by Fici and Lipt\'ak \cite{fici} and then further investigated by Burcsi \textit{et al.}~\cite{Burcsi, Burcsi2014}. 
The defining property of these binary words is that their prefixes  contain at least as many $1$'s as any of their factors of the same length.
Moreover, for any binary word $w$, there is a prefix normal word $w'$ of the same length such that for any length $k$, the maximal number  of $1$'s in a factor of length $k$ coincide for $w$ and $w'$. 
Such $w'$ is called the \emph{prefix normal form} of $w$.
The motivation behind the study of prefix normal words and prefix normal forms comes from a variation of the \emph{binary jumbled pattern matching} (binary JPM), which asks whether, for a text of length $n$ over a binary alphabet and two numbers $x$ and $y$, there exists a substring with $x$ $1$'s and $y$ $0$'s~\cite{butman}. 
While this problem has a simple efficient solution, its indexing variation, called \emph{indexing jumbled pattern matching} (IJPM) is not trivial. 
The binary IJPM problem asks whether one can preprocess a given text of length $n$ so that one can answer quickly $(x,y)$ queries.
Up to now, the best known construction of the index of size $\bigo(n)$ takes a time $\bigo(n^2/\log n)$~\cite{cicalese,moosa}. 
It is proven that prefix normal forms of the text can be used to construct this index~\cite{fici}.
Hence, the extensive study of prefix normal words and forms can yield improvements on the binary IJPM problem.

It turns out that prefix normal words appear in another completely different context from graph theory, when studying subtrees of caterpillar trees.
More precisely, let $L_G(i)$ be the maximum number of leaves that can be realized by an induced subtree of $G$ with exactly $i$ vertices.
We are interested in studying the properties of the finite sequence $L_G(i)_{i=0,1,\ldots,n}$, called the \emph{leaf sequence of $G$}, where $n$ is the number of vertices of $G$.
\begin{problem}[Leaf Realization Problem]\label{prob:graph}
Given a sequence of $n + 1$ natural numbers $(\ell_0,\ell_1,\ldots,\ell_n)$, does there exists a graph $G$ with $n$ vertices such that
$$(L_G(0), L_G(1), \ldots, L_G(n)) = (\ell_1, \ell_2, \ldots, \ell_n)?$$
\end{problem}

Problem~\ref{prob:graph} presents many similarities with other famous realization problems investigated more than 50 years ago.
For instance, considerable attention was devoted to the \emph{graph realization problem} \cite{Erdos-Gallai}, which consists in deciding whether a finite sequence of natural numbers $(d_1,\ldots,d_n)$ is the degree sequence of some labeled simple graph.
In the case where the answer is positive, the sequence $(d_1,\ldots,d_n)$ is called a \emph{graphic sequence}.
The problem was proven to be solvable in polynomial time \cite{Erdos-Gallai,Havel,Hakimi}.
In particular, it amounts to verify $n$ inequalities and whether the sum of degrees is even \cite{Erdos-Gallai}.
Several variations of the graph realization problem have been investigated, such as the \emph{bipartite realization problem} \cite{Gale,Ryser} and the \emph{digraph realization problem} \cite{Kleitman-Wang,Fulkerson,ChenR,Anstee,Berger}.

Although we do not succeed, in this paper, in presenting a complete answer to Problem~\ref{prob:graph}, we solve the following subproblem, which casts some light on the structure of leaf sequences and suggests that solving the  general Problem~\ref{prob:graph} is hard:
\begin{problem}[Leaf Realization Problem for Caterpillar Trees]\label{prob:cat}
Given a sequence of $n + 1$ natural numbers $(\ell_0,\ell_1,\ldots,\ell_n)$, does there exist a caterpillar tree $C$ of $n$ vertices such that
$$(L_C(0), L_C(1), \ldots, L_C(n)) = (\ell_1, \ell_2, \ldots, \ell_n)?$$
\end{problem}
It is worth mentioning that, contrary to the approaches used in \cite{Erdos-Gallai,Havel,Hakimi,Gale,Ryser,Kleitman-Wang,Fulkerson,ChenR,Anstee,Berger}, our solution relies on nontrivial concepts studied in combinatorics on words.
Indeed, if we consider the sequences of differences of consecutive elements of  leaf sequences, also called the discrete derivative of the  leaf sequences,  
 then we prove that,  for caterpillar graphs, the set
$$\Delta L_{\Cater} = \{ \Delta L_C: C \mbox{ is a caterpillar}\}$$
of discrete derivative of leaf sequences of caterpillar graphs 
is precisely the set of prefix normal words.
To prove this result, we introduce the notion of \emph{reading caterpillars} of a word. 
The link between binary words and their prefix normal forms is then mirrored in terms of their reading caterpillars.
Two words with the same prefix normal form are such that their reading caterpillars have the same leaf functions.
This is yet another example of the fruitful interaction between graph theory and combinatorics on words (see for example~\cite{Prufer-1918,Kitaev-2015}).

It is worth mentioning that our motivation in the study of induced subtrees having the maximum number of leaves comes from \cite{Blondin-FPSAC,graphs_arxiv}, where remarkable structures on regular lattices are presented.
Other similar problems have also been studied, such as maximum leaf spanning subtrees~\cite{payan,garey}, frequent subtrees mining~\cite{deepak} and induced subtrees~\cite{Erdos-Saks-Sos}.

The manuscript is organized as follows.
Preliminaries are recalled in Section~\ref{sec:preliminaries}.
We introduce the notions of leaf functions, leaf sequences and their discrete derivatives in Section~\ref{sec:fully-leafed}.
We develop some tools in Section~\ref{sec:caterpillar_sequences} that are useful for the proof of our main theorems.
Section~\ref{sec:words} and \ref{sec:forms} are devoted to our main theorems about the relationship between caterpillar graphs, prefix normal words and prefix normal forms.
We  conclude with some perspectives on future work in Section~\ref{sec:concl}.
Finally, \ref{sec:appendix} contains proofs omitted in the main part for the sake of readability.


\section{Preliminaries}\label{sec:preliminaries}

We start by recalling basic terminology on words. The reader is referred to Lothaire for a complete introduction~\cite{Lothaire}.
Let $\Sigma$ be a finite set called an \emph{alphabet} whose elements are called \emph{letters}.
A \emph{word} $w=w_1w_2\cdots w_n$ of length $n$ on the the alphabet $\Sigma$
 is a finite concatenation of $n$ letters $w_i\in\Sigma$.
A \emph{language} over $\Sigma$ is any set of words on $\Sigma$, either finite or infinite.
We denote by $\Sigma^*$ the language of all finite words over $\Sigma$ and by $\Sigma^n$ the language of all finite words of length  $n$.

A word $u$ is a \emph{factor} of $w$ when $w=pus$ is the concatenation of the words $p$, $u$ and $s$ for some words $p,s$. In that case, we say that $u$ \emph{occurs} in $w$.
If $p$ (respectively $s$) is the empty word, then $u$ is called a \emph{prefix} (resp. \emph{suffix}) of $w$. We denote by $\pref_i(w)$ (resp. $\suff_i(w)$) the unique prefix (resp. suffix) of $w$ of length $i$, and by $\Fact(w)$ (respectively $\Pref(w)$, $\Suff(w)$) the set of all its factors (resp. prefixes, suffixes).
We denote by $\Fact_i(w)$ the set $\Fact(w)\cap\Sigma^i$.

The \emph{alphabet} $\Alp(w)$ of a word $w$ is the set of letters occurring in $w$. The number of occurrences of the letter $a$ in $w$ is denoted by $|w|_a$. Two words $u$ and $v$ are called \emph{abelian equivalent} if $|u|_a = |v|_a$ for all letters $a\in\Sigma$.
Finally, the \emph{reversal} of a word $w$, denoted by $\tilde{w}$, is the word obtained by reading $w$ from right to left. If $w=w_1\cdots w_n$ with $w_i\in\Sigma$, then $\tilde{w}=w_n\cdots w_1$.
Given a sequence $a = (a_1,a_2,...,a_n)$, we use the notation $a[i]$ for the element $a_i$.

We also recall some definitions and notation about graph theory and we refer the reader to \cite{Diestel--2010} for an introduction to this subject. 
All  graphs considered in this text are simple and undirected. 
Let $G = (V,E)$ be a graph with vertex set $V$ and edge set $E$.
The \emph{degree} of a vertex $u$ is the number of vertices that are adjacent to $u$ and is denoted by $\deg(u)$.

We denote by $|G|$ or $|V|$ the total number of vertices of $G$ which is called the \emph{size} of $G$. 
For $U \subseteq V$, the \emph{subgraph of $G$ induced by $U$}, denoted by $G[U]$, is the graph $G[U] = (U, E \cap \mathcal{P}_2(U))$, where $\mathcal{P}_2(U)$ is the set of all subsets of size $2$ of $U$. 
An \emph{induced subtree of $G$} is a connected and acyclic induced subgraph of $G$, that is, a tree.

Let $T = (V,E)$ be a \emph{tree}.
We say that a vertex $u$ of $T$ with $\deg(u)=1$ is a \emph{leaf} of $T$ and we denote by $\leaf{T}$ the number of leaves of $T$.
Let $\leafset{V}$ be the set of leaves of $T$.
A tree $T$ is called a \emph{caterpillar} if the induced subgraph $T[V\setminus \leafset{V}]$ is a chain graph, i.e. if all leaves of $T$ are adjacent to a single central chain of $T$. We call this central chain $T[V\setminus \leafset{V}]$ the \emph{spine} of the caterpillar $T$.
The set of all caterpillars is denoted by $\Cater$.


\section{Fully leafed induced subtrees}\label{sec:fully-leafed}

We first recall the definition of \emph{leaf function} from \cite{graphs_arxiv}.

\begin{definition}[Leaf function, \cite{graphs_arxiv}]\label{D:leaf-function}
    Given a finite graph $G = (V,E)$, let $\mathcal{T}_G(i)$ be the family of all induced subtrees of $G$ with exactly $i$ vertices. The \emph{leaf function} of $G$, denoted by $L_G$, is the function with domain $\{0,1,2,\ldots,\size{G}\}$ defined by
    $$L_G(i)=\max\{\leaf{T} : T\in \mathcal{T}_G(i)\}.$$
    As is customary, we set $\max\emptyset = -\infty$.
    An induced subtree $T$ of $G$ with $i$ vertices is called \emph{fully leafed} when $\leaf{T} = L_G(i)$.
\end{definition}

\begin{example} \label{ex:leaf-function}
    Consider the graph $G$ depicted in Figure~\ref{fig:leaf_function}. Its leaf function is
    $$
    \begin{array}{c|ccccccccc}
    i &      0 & 1 & 2 & 3 & 4 & 5 & 6 & 7 & 8\\ \hline
    L_G(i) & 0 & 0 & 2 & 2 & 3 & 4 & 4 & 5 & -\infty\\
    \end{array}
    $$
    and the subtree induced by $\{1,2,3,4,6,8\}$ is fully leafed because it has $6$ vertices and $4$ leaves and $L_G(6) = 4$ leaves.
\end{example}

\begin{figure}[ht]
    \begin{center}
        \includegraphics[scale=1]{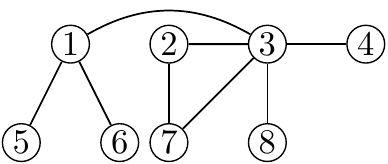}
    \end{center}
    \caption{A graph with vertex set $V=\{1,\ldots,8\}$}\label{fig:leaf_function}
\end{figure}

\begin{remark}
    For any simple graph $G$ with at least one vertex, we have $L_G(0) = 0$ since the empty tree has no leaf, and  $L_G(1) = 0$ because a single vertex is not a leaf. Finally, $L_G(2)=2$ in any graph $G$ with at least one edge.
    See \cite{graphs_arxiv} for more properties of the function $L_G$. 
\end{remark}

\begin{proposition}[\cite{graphs_arxiv}]\label{prop:non-decreasing}
    Let $G$ be a simple graph with $n\ge 3$ vertices.
    \begin{itemize}
        \item If $G$ is connected and non-isomorphic to $K_n$, the complete graph on $n$ vertices, then $L_G(3) = 2$.
        \item The sequence $(L_G(i))_{i=0,1,\ldots,n}$ is non-decreasing  if and only if $G$ is a tree.
    \end{itemize}
\end{proposition}

In what follows, we are interested in the internal structure of the words associated with the leaf functions.

\begin{definition}[Leaf sequence]\label{D:leaf-sequence}
    Let $\ell$ be a finite sequence in the set $\N\cup\{-\infty\}$. We say that $\ell$ is a \emph{leaf sequence} if there exists a simple graph $G$ such that the sequence of values of its leaf function $L_G$ is equal to $\ell$, i.e. if  $\ell = (L_G(i))_{i=0,1,2,\ldots, \size{G}}.$
\end{definition}

A useful concept in the investigation of leaf sequences is the associated word of its first differences, also called the \emph{discrete derivative}.
\begin{definition}[Leaf word]\label{D:leaf word}
    Let $G$ be a simple graph with $n$ vertices and $L_G$ its associated leaf function.
    The \emph{leaf word} of $G$, denoted by $\Delta L_G$, is the word on the alphabet $\Z \cup \{\infl\}$ with $i$-th letter given by
    $$\Delta L_G(i) = L_G(i+3) - L_G(i+2),$$
    for $i = 1,2,\ldots,n-3$.
    We set $L_G(i+3) - L_G(i+2) = \infl$ whenever $L_G(i+2) = -\infty$ or $L_G(i+3) = -\infty$.
\end{definition}

\begin{remark}
    Recall that $L_G(0) = 0$, $L_G(1) = 0$, $L_G(2) = 2$ for any graph with at least one edge and that $L_G(3) = 2$ for any connected graph non isomorphic to a complete graph. Therefore, for most connected graphs, the sequence of differences $(L_G(i)-L_G(i-1))_{i=1,\ldots,n}$ of consecutive values of $L_G$ always starts with the prefix $020$. To avoid repeating this information, we have chosen to erase the first three letters in the sequence of first differences and shift indices by $3$ units from left to right  in the above definition of leaf word. 
\end{remark}

\begin{example}
    \label{example:LW}
    Consider the graph $G$  represented in Figure~\ref{fig:ex_leaf_word}.  The prefix $020$ in the sequence of differences  is omitted 
  and the leaf word of $G$ is  $\Delta L_G = 110101$.
    \begin{figure}[ht]
        \begin{center}
            \includegraphics{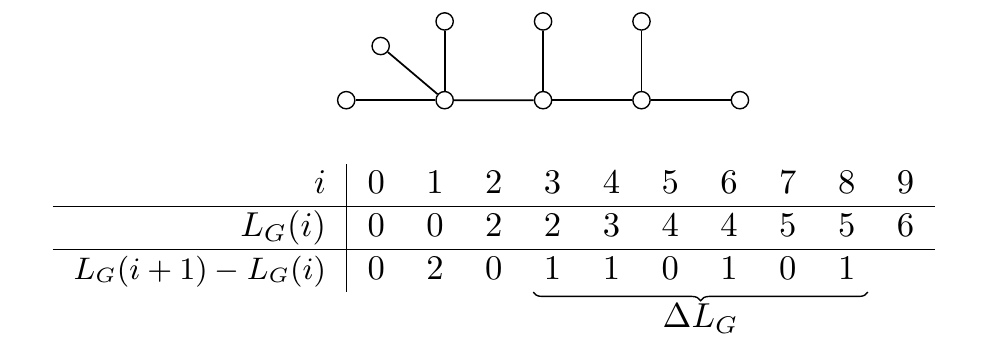}
        \end{center}
        \caption{A graph $G$, its leaf function $L_G$, its sequence of differences  and its leaf word $\Delta L_G$.}\label{fig:ex_leaf_word}
    \end{figure}
    
\end{example}

A first step to  answer Problem~\ref{prob:graph} consists in describing the admissible alphabets of leaf words.
\begin{lemma}\label{L:alph01} The set $S=\{1,0,-1,-2,\dots,\infl\}$ is the smallest set such that for any graph $G$ with at least $3$ vertices we have  $\Alp(\Delta L_G) \subseteq S$.
\end{lemma}

\begin{proof}
    Let $L_G$ be the leaf function of a simple graph $G$ with at least $3$ vertices and let $i \geq 3$. Obviously,  $1$, $0$ and $\infl$ are possible values of $L_G(i+1) - L_G(i)$ as shown in the leaf function of Example~\ref{ex:leaf-function}.
    Additionally, the difference $L_G(i+1)- L_G(i)$ may take any negative integer value. 
    Indeed, consider the wheel graph $W_n$ with $n+1$ vertices. Its leaf function is

    $$L_{W_n}(i) = \begin{cases}
    0,       & \mbox{if $i = 0, 1$;} \\
    2,       & \mbox{if $i = 2$;} \\
    i-1,     & \mbox{if $3 \leq i \leq \lfloor\frac{n}{2}\rfloor+1$;} \\
    2,       & \mbox{if $\lfloor\frac{n}{2}\rfloor+2 \leq i \leq n - 1$;} \\
    -\infty, & \mbox{if $n \le i \le n+1$;}
    \end{cases}$$
    as shown in \cite{graphs_arxiv}.  Therefore $\Delta L_{W_{2k+4}}$ contains the letter $-k$ for any $k\geq 1$.  The particular case of $W_{10}$ is depicted in Figure~\ref{fig:wheel} where we observe that \\
 $L_{W_{10}}(7)- L_{W_{10}}(6)=-3$.
    
    It remains to show that for any graph $G$, the positive integers $k>1$ cannot occur in $\Delta L_G$. Arguing by contradiction, assume that $G$ is a graph with $n$ vertices such that $\Delta L_G$ contains some letter $k \in \mathbb{N} \setminus \{0,1\}$. By Definition~\ref{D:leaf word}, we have $L_G(i+1)- L_G(i) = k>1$ for some $3\leq i<n$. This means that there exists an induced subtree $T$ of $G$ with $i+1$
    vertices and $L_G(i)+k$ leaves. Let $T'$ be a  subgraph obtained by removing one leaf from $T$. Then $T'$ is an induced subtree with $i$ vertices and has at least $L_G(i) + k - 1 > L_G(i)$ leaves, contradicting the definition of $L_G(i)$.
\end{proof}

\begin{figure}[ht]
    \begin{center}
        \includegraphics{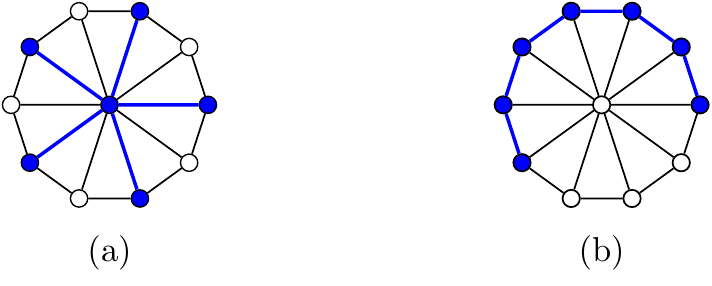}
    \end{center}
    \caption{The wheel graph $W_{10}$ with two fully leafed induced subtrees of respective size (a) $6$ and (b) $7$ in blue.}\label{fig:wheel}
\end{figure}

An elementary yet useful observation is that the alphabet of the leaf word indicates whether the associated graph is a tree. The following lemma is an immediate consequence of Lemma~\ref{L:alph01} and Proposition~\ref{prop:non-decreasing}.

\begin{lemma}~\label{prop:alph-graph}
    Let $G$ be a graph with at least 3 vertices. Then $\Alp(\Delta L_G) \subseteq \{0,1\}$ if and only if $G$ is a tree.
\end{lemma}

Now it follows directly from Definition~\ref{D:leaf word} and Lemma~\ref{prop:alph-graph} that, for a tree $T$,
\begin{equation}\label{EQ:LDL}
L_T(j) - L_T(i) = |\suff_{j - i}(\pref_{j-3}(\Delta L_T))|_1
\end{equation}
for $3 \leq i \leq j \leq n$.

For instance in the graph $G$ of Example~\ref{example:LW} showed in Figure \ref{fig:ex_leaf_word}, we verify Equation~\eqref{EQ:LDL} with $i = 4$ and $j = 7$ as $L_G(7)-L_G(4)= 2 = |101|_1 = |\suff_{3}(\pref_{4}(\Delta L_G))|_1$.

At this stage, we are not able to provide a complete answer to Problem~\ref{prob:cat} for general graphs.
 However, restricting our attention to caterpillar graphs leads to  an interesting connection with the so-called \emph{prefix normal words}. To show this connection, we need to introduce the notion of caterpillar sequences.


\section{Caterpillar sequences}\label{sec:caterpillar_sequences}

For practical purposes, it is convenient to represent caterpillar graphs by sequences of natural numbers.
\begin{definition}~\label{D:dir_cater}
    A \emph{caterpillar sequence} $S=(s_1,\ldots,s_k)$ of length $k$ is a sequence of $k$ non negative integers such that $s_1, s_k \geq 1$ and, if $k = 1$, then $s_1 \geq 2$.
   
    The \emph{size} $\size{S}$ of $S$ and its \emph{number of leaves} $\leaf{S}$ are respectively defined by
    $$\size{S} = k + \sum_{i=1}^k s_i \text{ and }
    \leaf{S} = \sum_{i=1}^k s_i.$$
   The \emph{reversal} of $S$ is the caterpillar sequence $\widetilde{S} = (s_k,\ldots,s_1)$, that we also denote by $S\widetilde{\phantom{ab}}$ with the exponent notation for typographic reasons.
   We denote by $\Sub$ the set of all caterpillar sequences.
\end{definition}

We choose the expression ``caterpillar sequence'' to explain the fact that to each caterpillar sequence $S = (s_1,\ldots,s_k)$ corresponds, up to isomorphism, a unique caterpillar graph $C$ with spine $(v_1,\ldots,v_k)$ such that the number of leaves adjacent to $v_i$ is $s_i$ for all $i\in\{1,\ldots,k\}$.
Therefore,
$$\size{S} = \size{C} \quad \text{and} \quad \leaf{S} = \leaf{C}.$$
Conversely, to each caterpillar graph $C$ corresponds one caterpillar sequence $S$ obtained by arbitrarily choosing one among the two orientations of its spine, say $(v_1,\cdots,v_k)$, and by setting $s_1 = \deg(v_1) - 1$, $s_k = \deg(v_k) - 1$ and
$$s_i=\deg(v_i)-2, \quad \mbox{for all $i\in\{2,\ldots,k-1\}$},$$
if $k > 1$, and simply $s_1=\deg(v_1)$ if $k = 1$.

Next, we define the function $\spine$ on caterpillar sequences as the function that produces the sequence of degrees of the corresponding caterpillars :

\begin{align*}\label{DEF:f}
\spine(s_1,\ldots,s_k)=\begin{cases}
(s_k) &\mbox{ if } k=1,\\
(s_1+1,s_2+2,\ldots , s_{k-1}+2,s_k+1) &\mbox{ if } k>1.
\end{cases}
\end{align*}

Observe that in the definition of $\spine$, we avoid double parentheses and write $\spine(s_1,\dots,s_k)$ instead of $\spine((s_1,\dots, s_k))$.

\begin{definition}\label{D:subsequence}
    Given two caterpillar sequences $S=(s_1,\ldots,s_k)$ and $S' = (s'_1,\ldots,s'_{k'})$ of respective
    lengths $k$ and $k'$ with $1\leq k'\leq k$, we say that $S'$ is a \emph{caterpillar subsequence} of $S$, and we write $S' \preceq S$, if there exists an integer $i\in\{0,\ldots,k-k'\}$ such that
    $$\spine(S')[j] \leq \spine(S)[j+i], \quad \mbox{for all $j \in \{1,\ldots,k'\}$.}$$
\end{definition}

Finally, given any caterpillar sequence $S$, we denote by $L_S$ the \emph{leaf function} of its associated caterpillar graph, so that Definition~\ref{D:leaf-function} translates as
\begin{equation}\label{EQ:leaf-function-seq}
L_S(i) = \max\{ \leaf{S'} : S' \preceq S, \size{S'} = i\}.
\end{equation}

The next observation, whose proof can be found in \ref{sec:appendix}, follows directly from Definition~\ref{D:subsequence}.
\begin{proposition}\label{P:poset}
    The relation $\preceq$ is a partial order on the set $\Sub$ of all caterpillar sequences.
\end{proposition}
One may notice that the caterpillar sequences $(1,1)$ and $(3)$ are both covered by $(1,2)$ and $(2,1)$ with respect to the order $\preceq$  (see Figure~\ref{fig:poset_caterpillar})
so that the poset $(\Sub,\preceq)$ is not a lattice.
\begin{figure}[ht]
    \centering
    \includegraphics[scale=0.28]{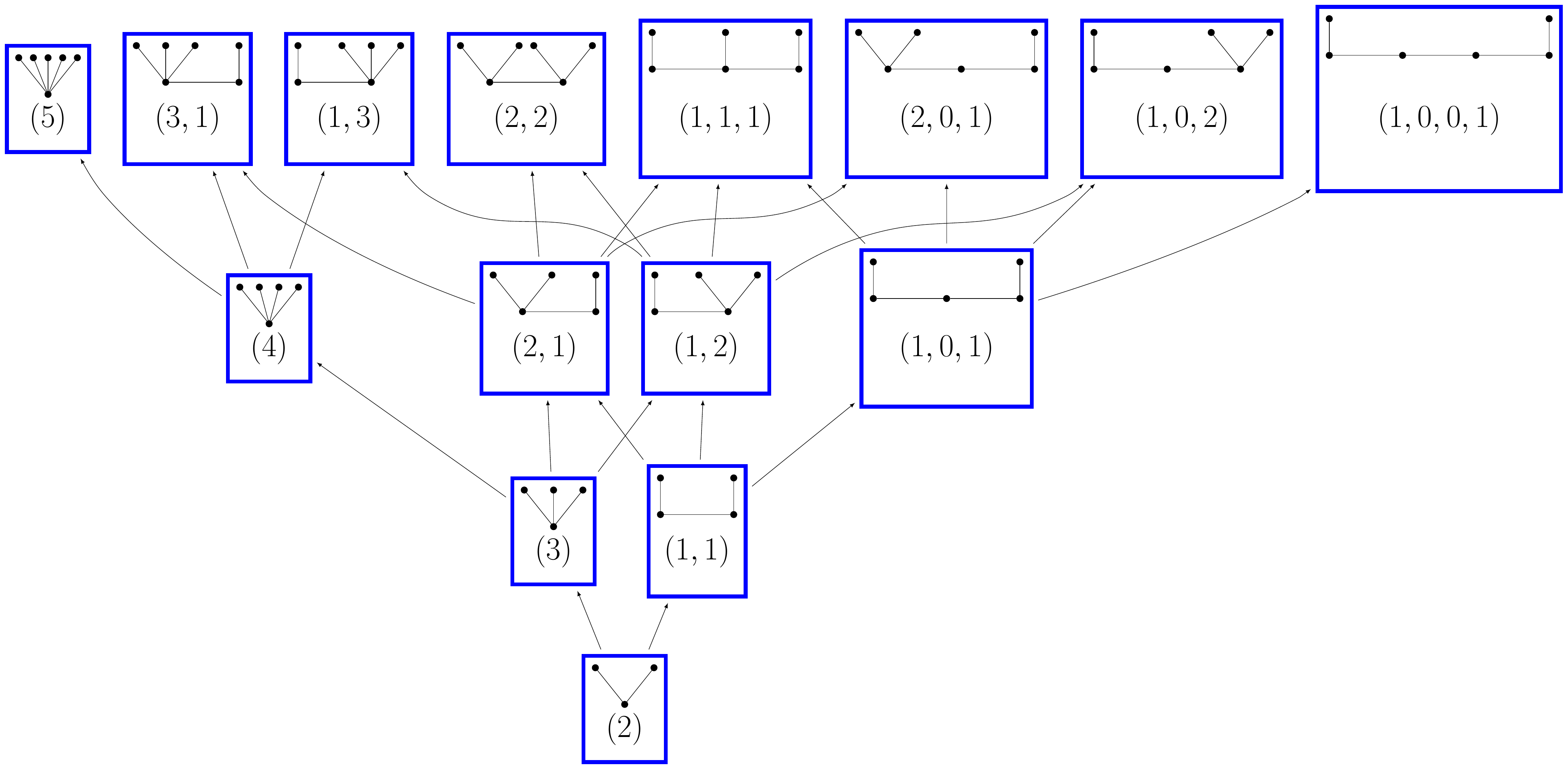}
    \caption{Hasse diagram of caterpillar sequences of size up to 6 of $(\Sub,\preceq)$} \label{fig:poset_caterpillar}
\end{figure}

In fact, caterpillar subsequences are precisely caterpillar sequences of induced subcaterpillars.
\begin{definition} \label{def:fully-leafed-subseq}
    A caterpillar subsequence $S'$ of $S$ is called \emph{fully leafed} if its corresponding tree is a fully leafed induced subtree of the tree associated to $S$. 
\end{definition}

Leftmost and rightmost caterpillar subsequences are of particular interest.
Let $S = (s_1,\ldots, s_k)$ be a caterpillar sequence. Given $i \in \{3,4,\ldots,|S|\}$, the \emph{left caterpillar subsequence of size $i$} of $S$ is defined recursively by
\begin{equation}\label{EQ:left}
\Left_i(S) = \begin{cases}
S,
& \mbox{if $i = |S|$;} \\
\Left_i(s_1,\ldots,s_{k-1},s_k-1),
& \mbox{if $i < |S| \mbox{ and } s_k \geq 2$;} \\
\Left_i(s_1,\ldots,s_{k-1} + 1),
& \mbox{if $i < |S| \mbox{ and } s_k = 1$.}
\end{cases}
\end{equation}

The \emph{right caterpillar subsequence of size $i$} of $S$ is defined similarly by
\begin{equation}\label{EQ:right}
\Right_i(S) = \begin{cases}
S,
& \mbox{if $i = |S|$;} \\
\Right_i(s_1 - 1,s_2,\ldots,s_k),
& \mbox{if $i < |S| \mbox{ and } s_1 \geq 2$;} \\
\Right_i(s_2 + 1,\ldots,s_k),
& \mbox{if $i < |S| \mbox{ and }  s_1 = 1$.}
\end{cases}
\end{equation}

The following observations are immediate. See \ref{sec:appendix} for the proofs.
\begin{lemma}\label{L:subs}
    Let $S$ be a caterpillar sequence and $3 \leq i \leq |S|$. Then
    \begin{enumerate}[\rm(i)]
        \item $\Left_i(S) \preceq S$;
        \item $\Right_i(S) \preceq S$;
        \item $\Left_i(\widetilde{S}) = \widetilde{\Right_i(S)}$.\label{L:subsiii}
    \end{enumerate}
\end{lemma}

 $\Left_i(S)$ and $\Right_i(S)$ are unique sequences (see \ref{sec:appendix} for the proof).
\begin{lemma}\label{L:alpha-beta}
Let $S = (s_1,\ldots,s_k)$ be a caterpillar sequence of length $k$, and $3 \leq i \leq |S|$. Then there exist unique integers $a$ and $\alpha$ such that $\Left_i(S) = (s_1,\dots, s_a,\alpha)$, satisfying the relations

\begin{equation}
0 \leq a \leq k - 1, \quad
1 \leq \alpha \leq s_{a+1} + 1 \quad \text{and} \quad
i = \sum_{m = 1}^a (s_m + 1) + (\alpha + 1).\label{EQ:alpha}
\end{equation}
In the same way, there exist unique integers $b$ and $\beta$ such that
$\Right_i(S) = (\beta, s_b, \dots, s_k)$, where
\begin{equation}
2 \leq b \leq k+1, \quad
1 \leq \beta \leq s_{b-1} + 1 \quad \text{and} \quad
i = \sum_{m = b}^k (s_m + 1) + (\beta + 1).\label{EQ:beta}
\end{equation}
\end{lemma}

\begin{remark}
In Lemma~\ref{L:alpha-beta}, when $a = 0$, we abuse the notation and define $(s_1,\ldots,s_a,\alpha) = (\alpha)$. Similarly, if $b = 0$, then $(\beta,s_b,\ldots,s_k) = (\beta)$.
\end{remark}

\begin{example}\label{ex:subsequences}
    Consider the caterpillar graph $C$, depicted in Figure~\ref{fig:subsequences}, with size $\size{C}=16$ 
    and number of leaves  $\leaf{C}=10$. Its caterpillar sequence is $S=(3,0,2,4,0,1)$ with $\size{C}=\size{S}$.
    The sequence $(1,1,4,0,1)$ is a caterpillar subsequence of $S$ with associated subcaterpillar shown in Figure~\ref{fig:subsequences} $(b)$.
    It is not a right caterpillar subsequence because one leaf adjacent to the third leftmost inner vertex is missing. The sequence $(6)$ is a caterpillar subsequence of $S$ and has a unique associated subcaterpillar of $C$ shown in Figure~\ref{fig:subsequences} $(a)$.
    The left and right caterpillar subsequences of $S$ of size $6$ are respectively
    $$\Left_6(S) = (3,1)\text{ and }\Right_6(S) = (2,0,1)$$
    and their associated graphs are shown in Figures~\ref{fig:subsequences} $(c)$ and $(d)$.
\end{example}

\begin{figure}[ht]
    \begin{center}
        \includegraphics{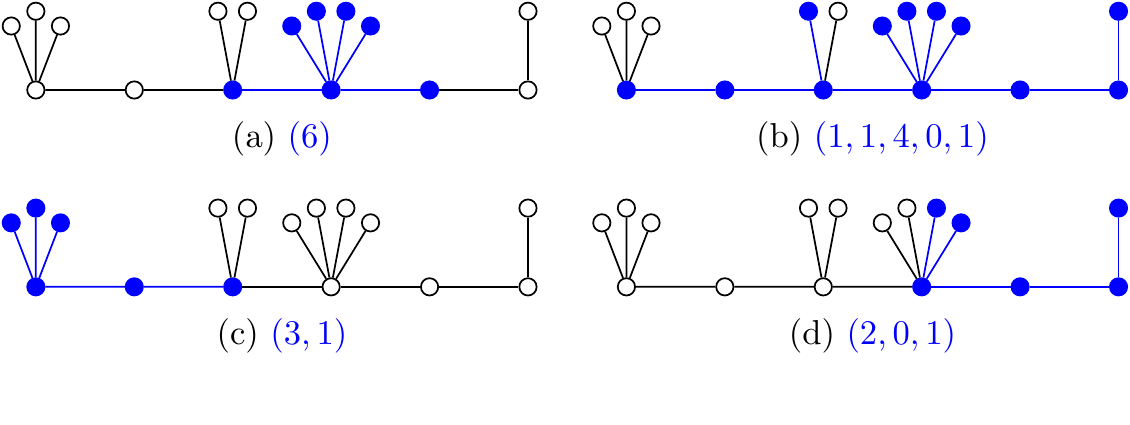}
    \end{center}
    \caption{
        The  graph associated to the caterpillar sequence $(3,0,2,4,0,1)$  and some induced subtrees, highlighted in blue.
        (a) The associated subtree of the caterpillar subsequence $(6)$.
        (b) The associated subtree of the caterpillar subsequence $(1,1,4,0,1)$.
        (c) The associated subtree of the  left caterpillar subsequence of size $6$.
        (d) The associated subtree of the right caterpillar subsequence of size $6$.
    }
    \label{fig:subsequences}
\end{figure}

We now introduce a useful operation called the \emph{graft} of caterpillar sequences.

\begin{definition}\label{D:graft}
    Let $S = (s_1,s_2,\dots,s_k)$ and $S' = (s'_1,s'_2,\dots,s'_l)$ be two caterpillar sequences. The graft of $S$ and $S'$, denoted by $S\diamond S'$, is the caterpillar sequence
    \begin{eqnarray*}
        S\diamond S' = (s_1,s_2,\dots, s_{k-1},s_k +s'_1 -2,s'_2,\dots,s'_{k'}).
    \end{eqnarray*}
\end{definition}

As an example, we have  $(4,1)\diamond (3,0,1)=(4,2,0,1)$. Figure~\ref{fig:graft} shows how the graft of sequences is interpreted on the corresponding caterpillars.

\begin{figure}[ht]
    \begin{center}
        \includegraphics[scale=1]{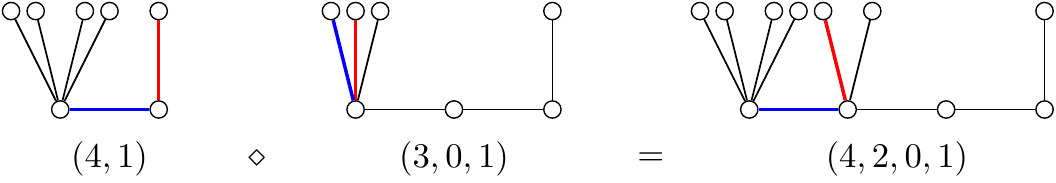}
    \end{center}
    \caption{The graft of the caterpillar sequences $(4,1)$ and $(3,0,1)$ and the corresponding caterpillar graphs. The blue (respectively red) edges on left hand side are merged on the right hand side.}
    \label{fig:graft}
\end{figure}

It is straightforward to prove that $(\Sub, \diamond)$ is a non-commutative monoid with identity $(2)$. Moreover this monoid is isomorphic to $(\Sigma^*, \cdot)$, the monoid of words with concatenation.
Also, $S$, $S'\preceq S\diamond S'$, for any caterpillar sequences $S$ and $S'$.

The maps $\size{\cdot}$ and $\leaf{\cdot}$ on caterpillar sequences
are compatible with the graft operation in the following sense. Given two caterpillar sequences $S$ and $S'$, we have
\begin{alignat}{2}\label{EQ:n1}
\leaf{S\diamond S'} & = \leaf{S}+\leaf{S'}-2, \\
\label{EQ:n}
\size{S  \diamond S'} & = \size{S}+\size{S'}-3.
\end{alignat}

Also, a straightforward computation shows that for any caterpillar sequences $S$ and $S'$:
\begin{equation}\label{eq:graft-tilde}
    \widetilde{S \diamond S'} = \widetilde{S'} \diamond \widetilde{S}.
\end{equation}

The graft operation is useful for the decomposition of a caterpillar sequence as a ``product'' of  smaller sequences.
\begin{lemma}~\label{l:graft}
    Let $S$ be a caterpillar sequence. Then, for any integer $i \in \{3,4,\ldots,|S|\}$,
    $$S = \Left_i(S) \diamond \Right_{|S|+3-i}(S).$$
\end{lemma}

\begin{proof}
Let $n = |S|$ and write $S = (s_1,\dots, s_k)$.
From Lemma~\ref{L:alpha-beta}, we have
$$\Left_i(S) = (s_1,\ldots,s_a,\alpha) \quad \text{and} \quad \Right_{n+3-i}(S) = (\beta,s_b,\ldots,s_k),$$
where $0 \leq a \leq k - 1$, $1 \leq \alpha \leq s_{a+1}+1$, $2 \leq b \leq k+1$, $1 \leq \beta \leq s_{b-1}+1$ and
$$i = \sum_{m=1}^a (s_m+1) + (\alpha + 1) \quad \text{and} \quad n+3-i = \sum_{m=b}^k (s_m + 1) + (\beta + 1).$$
Summing up those two last equations yields
\begin{equation}\label{EQ:ab}
n + 3 = \sum_{m=1}^a s_m + \sum_{m=b}^k s_m + (a + (k - b + 1)) + (\alpha + \beta) + 2.
\end{equation}
We claim that $a + 1 = b - 1$. Arguing by contradiction, assume first that $a + 1 > b - 1$. Then Equation~\eqref{EQ:ab} implies
$$n + 3 \geq \left(\sum_{m=1}^k s_m + k\right) + (a - b + 1) + (\alpha + \beta + 2) > n + \alpha + \beta + 1,$$
so that $\alpha + \beta < 2$, which is impossible.
Next, assume that $a + 1 < b - 1$.
Using again Equation~\eqref{EQ:ab}, we obtain
\begin{align*}
n + 3 &\leq \left(\sum_{m=1}^k s_m + k\right) - (s_{a+1} + s_{b-1}) + (a - b + 1) + (\alpha + \beta + 2) \\
&< n - s_{a+1} - s_{b-1} + \alpha + \beta + 1.
\end{align*}
Thus $\alpha + \beta > s_{a+1} + s_{b-1} + 2$, which is also impossible.

Therefore,
$$n + 3 = n - s_{a+1} + \alpha + \beta + 1,$$
which implies $s_{a+1} = \alpha + \beta - 2$. Hence
\begin{alignat*}{2}
  \Left_i(S) \diamond \Right_{n+3-i}(S)
    & = (s_1,\ldots,s_a,\alpha) \diamond (\beta,s_b,\ldots,s_k) \\
    & = (s_1,\ldots,s_a,\alpha+\beta-2,s_b,\ldots,s_k) \\
    & = (s_1,\ldots,s_a,s_{a+1},s_{a+2},\ldots,s_k) \\
    & = S,
    \end{alignat*}
    concluding the proof.
\end{proof}

\begin{remark}\label{R:factor-tree}
The graft of $\Left_i(S)$ and $\Right_{n+3-i}(S)$ provides a factorization of any caterpillar into two smaller caterpillars of arbitrary size $i$ and $n+3-i$, for any $3 \leq i \leq n$, but it does not seem possible to extend this factorization to the larger class of trees. For example, we do not see how one could ``factor'' the tree illustrated in Figure~\ref{fig:ternary_tree} as the graft of  one tree  of size $5$ and one tree of size $11$.
\end{remark}

\begin{figure}[ht]
    \begin{center}
        \includegraphics{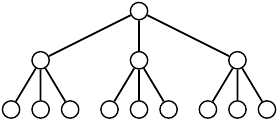}
    \end{center}
    \caption{A tree with leaf word non prefix normal.}
    \label{fig:ternary_tree}
\end{figure}

Caterpillars can be built naturally by reading binary words from left to right.
\begin{definition}\label{D:reading}
    We define the \emph{reading caterpillar sequence} $\RC(w)$ of a binary word $w$ recursively on $|w|$ as follows. The reading caterpillar sequence $\RC(\varepsilon)$ of the empty word is the sequence $(2)$ corresponding to a simple chain on three vertices.
    Let $ua$ be a binary word with $u\in\{0,1\}^*,a\in\{0,1\}$ and let $\RC(u)=(r_1,\ldots,r_k)$ be the reading caterpillar sequence of $u$. Then the reading caterpillar sequence of $ua$ is
    \begin{equation}\label{EQ:reading-graft}
    \RC(ua) = \begin{cases}
    \RC(u)\diamond (1,1), & \mbox{if $a = 0$;} \\
    \RC(u)\diamond (3),   & \mbox{if $a = 1$.}
    \end{cases}
    \end{equation}
\end{definition}

Note that Equation~\eqref{EQ:reading-graft} is equivalent to
\begin{equation}\label{EQ:reading}
    \RC(ua) = \begin{cases}
    (r_1,\ldots,r_k -1,1), & \mbox{if $a = 0$;} \\
    (r_1,\ldots,r_k +1),   & \mbox{if $a = 1$.}
    \end{cases}
\end{equation}

The following observations are immediate consequences of Definition~\ref{D:reading} (see \ref{sec:appendix} for the proof).

\begin{lemma}\label{L:cwt}
    Let $u, v$ and $w$ be binary words on $\{0,1\}$. Then
    \begin{enumerate}[\rm(i)]
        \item $\RC(uv) = \RC(u) \diamond \RC(v)$; \label{L:cwti}
        \item $\RC(\widetilde{w}) = \widetilde{\RC(w)}$\label{L:cwtii}.
    \end{enumerate}
\end{lemma}
Lemma~\ref{L:cwt}~\eqref{L:cwti} shows that $RC:\Sigma^*\to \Sub$  is a monoid morphism. We can easily verify that it is in fact an isomorphism.

\begin{lemma}\label{L:reading}
    Let $w$ be a binary word on $\{0,1\}$, $a \in \{0,1\}$ and $3 \leq i \leq |w| + 3$.
    Then
    \begin{enumerate}[\rm(i)]
        \item $\Left_i(\RC(w)) = \RC(\pref_{i-3}(w))$;\label{L:ri}
        \item $\Right_i(\RC(w)) = \RC(\suff_{i-3}(w))$;\label{L:rii}
        \item $\size{\RC(w)} = |w|+3$; \label{L:size}
        \item $\leaf{\RC(wa)} = \leaf{\RC(w)} + a$;\label{L:riii}
        \item $\leaf{\RC(w)} = |w|_1 + 2$;\label{L:riv}
        \item $\leaf{\Left_i(\RC(w))} = |\pref_{i-3}(w)|_1+2$;\label{L:rV}
        \item $\leaf{\Right_i(\RC(w))} = |\suff_{i-3}(w)|_1+2$.\label{L:rvi}
    \end{enumerate}
\end{lemma}

The proof of Lemma~\ref{L:reading} is found in~\ref{sec:appendix}.

\section{Caterpillars and prefix normal words}\label{sec:words}

We are now in a position to answer Problem~\ref{prob:cat} and to explicitly describe the connection between caterpillars and prefix normal words.

\begin{definition}[\cite{fici}]
    A binary word $u$ on the alphabet $\{0,1\}$ is called \emph{prefix normal} if for any prefix $p \in \Pref(u)$ and any factor $f \in \Fact(u)$, the condition $|p| = |f|$ implies $|p|_1 \geq |f|_1$. We denote by $\PNW$ the set of prefix normal words.
\end{definition}

In what follows, we see that the restriction $\RC|_{\PNW}$ of constructing the reading caterpillar sequence of a prefix normal word
is the right inverse of the operation
$$\begin{array}{ccccccccccccc}
  \Delta L & : & \Sub & \to     & \PNW \\
           &   & S    & \mapsto & \Delta L_S
\end{array}$$
of constructing the leaf word of a caterpillar sequence.
In other words, we show in the next pages that
$$\Delta L \circ \RC|_{\PNW} = I_{\PNW}.$$
However, notice that $\RC \circ \Delta L$ is different from the identity $I_{\Sub}$.

More precisely, we prove the following theorem.
\begin{theorem}\label{T:caterpillars}
    Let $L :\{0,1,2,\dots, n\} \to \N$ be a function such that $L(0) = 0$, $L(1) = 0$, $L(2) = 2$ and $L(3) = 2$. Then there exists a caterpillar $C$ such that $L=L_C$  if and only if $\Delta L$ is a prefix normal word.
\end{theorem}

The proof of Theorem~\ref{T:caterpillars} is divided into Propositions~\ref{prop:pn_implies_cat} and~\ref{prop:cat_implies_pn}.

\begin{proposition}\label{prop:pn_implies_cat}
    Let $w$ be a prefix normal word and $3 \leq i \leq |w| + 3$. Then
    \begin{enumerate}[\rm(i)]
    	\item $ L_{\RC(w)}(i) = \leaf{\Left_i(\RC(w))};$ \label{prop:pn_part+_1}
    	\item $\Delta L_{\RC(w)} = w$.
    \end{enumerate}
\end{proposition}

\begin{proof}
		(i) It is immediate that $L_{\RC(w)}(i) \geq \leaf{\Left_i(\RC(w))}$, since $\Left_i(\RC(w))$ is a caterpillar subsequence of size $i$ of $\RC(w)$.
		Thus we  only need to prove that 
		$$L_{\RC(w)}(i) \leq \leaf{\Left_i(\RC(w))}.$$
		We proceed by induction on $|w|$.
		
		\textsc{Basis}. If $|w| = 0$, then $w = \varepsilon$ and $i = 3$ 
		so that
		\begin{align*}
		L_{\RC(w)}(i) 
		& = L_{\RC(\varepsilon)}(3) = 2 \\
		& = \leaf{\Left_3(\RC(\varepsilon))} = \leaf{\Left_i(\RC(w))} \\
		& \leq \leaf{\Left_i( \RC(w))}.
		\end{align*}
		
		\textsc{Induction.} Since $|w| > 0$, there exist a word $u$ and a letter $a$ such that $w = ua$.
Assume first that $i = |w| + 3$. Then Equation~\eqref{EQ:leaf-function-seq} implies $L_{\RC(w)}(i) = \leaf{\Left_i(\RC(w))}$ since $\Left_i(\RC(w))=\RC(w)$ is the only caterpillar subsequence of size $|w| + 3$ of $\RC(w)$.

		It remains to consider the case $i < |w| + 3$. Arguing by contradiction, assume that there exists a caterpillar subsequence $S$ of $\RC(w)$ of size $i$ such that $\leaf{S} > \leaf{\Left_i(\RC(w))}$.
		If $S \preceq \RC(u)$, then
		\begin{alignat}{2}
		\leaf{S} & > \leaf{\Left_i(\RC(w))} \label{EQ:deb} \\
		& = \leaf{\Left_i(\RC(u))}
		&& \qquad \mbox{(since $i < |w| + 3$)} \\
		& = L_{\RC(u)}(i)
		&& \qquad \mbox{(by induction hypothesis)}\label{EQ:fin}
		\end{alignat}
		contradicting the maximality of $L_{\RC(u)}(i)$. Hence, $S \not\preceq \RC(u)$.
		
		Next, assume that $S = \Right_i(\RC(w))$. Then we have
		\begin{align*}
		|\pref_{i-3}(w)|_1 + 2 &= \leaf{\Left_i(\RC(w))} \\
		& < \leaf{S} = \leaf{\Right_i(\RC(w))} = |\suff_{i-3}(w)|_1 + 2,
		\end{align*}
		i.e.\ $|\pref_{i-3}(w)|_1 < |\suff_{i-3}(w)|_1$, contradicting the assumption that $w$ is prefix normal.
		
		To conclude, assume that $S$ is neither a caterpillar subsequence of $\RC(u)$ nor a right caterpillar subsequence of $\RC(w)$.
Let $\RC(w) = (r_1,r_2,\ldots,r_k)$ and $S = (s_1,s_2,\ldots,s_{k'})$.
Since $S \preceq \RC(w)$ but $S \not\preceq \RC(u)$, we have $k' \leq k$ and $s_j \leq r_{j+k-k'}$ for $j = 1,2,\ldots,k'$.
Let $j$ be the largest index such that $s_j < r_{j+k-k'}$ (such an index $j$ exists since $S \neq \Right_i(\RC(w))$) and let
		$$S' = (s_1,s_2,\ldots,s_{j-1},s_j + 1,s_{j+1},\ldots,s_{k'} - 1).$$
		Clearly, $S' \preceq \RC(u)$,
		but $\leaf{S'} = \leaf{S}$ so that a  sequence of relations similar to relations \eqref{EQ:deb}--\eqref{EQ:fin} obtained by replacing $S$ by $S'$ leads to a contradiction, concluding the proof.
    
   \noindent (ii) For $1 \leq i \leq |w|$, we have
    \begin{align*}
    \Delta &L_{\RC(w)}(i) \\
    & = L_{\RC(w)}(i+3) - L_{\RC(w)}(i+2)
    && \qquad \mbox{(by Definition~\ref{D:leaf word})} \\
    & = \leaf{\Left_{i+3}(\RC(w))} - \leaf{\Left_{i+2}(\RC(w))}
    && \qquad \mbox{(by Proposition~\ref{prop:pn_implies_cat}~\eqref{prop:pn_part+_1})} \\
    & = (|\pref_{i}(w)|_1 + 2) - (|\pref_{i-1}(w)|_1 + 2)
    && \qquad \mbox{(by Lemma~\ref{L:reading} \eqref{L:rV})} \\
    & = w_{i},
    \end{align*}
    as claimed.
\end{proof}

The following simple observation about non prefix normal words is the key to proving the second part of Theorem~\ref{T:caterpillars}.

\begin{lemma}\label{lem:npn}
    Let $w$ be a binary word on $\{0,1\}$ that is non prefix normal. Then there exist two abelian equivalent words $u$ and $u'$, such that $u0 \in \Pref(w)$ and $1u' \in \Fact(w)$.
\end{lemma}

\begin{proof}
    Let $w$ be a non prefix normal word. then there exists at least one prefix $p$ and one factor $f$ of $w$ of the same length such that $|p|_1 < |f|_1$.
    Without loss of generality,  assume that $|p|$ and $|f|$ are as small as possible. Let $p = ua$ and $f = bu'$ for some letters $a,b\in\{ 0,1\}$. Since $|p|$ and $|f|$ are minimal, we have $|u|_1 \geq |u'|_1$.
    Therefore,
    $$|u|_1 + a = |ua|_1 = |p|_1 < |f|_1 = |bu'|_1 = |u'|_1 + b \leq |u|_1 + b,$$
    which can only be verified when $a = 0$ and $b = 1$, in which case $|u|_1 = |u'|_1$.
\end{proof}

We are now ready to describe the structure of the leaf sequences of caterpillars.
\begin{proposition}\label{prop:cat_implies_pn}
    Let $C$ be a caterpillar. Then $\Delta L_C$ is prefix normal.
\end{proposition}

\begin{proof}
    We proceed by contradiction and assume that $\Delta L_C$ is not prefix normal.
    By Lemma~\ref{lem:npn}, there exist two words $p$ and $f$, with $|p|_1 = |f|_1$ such that $p0 \in \Pref(\Delta L_C)$ and $1f \in \Fact(\Delta L_C)$.
    Let $i - 3$ be the index of the last letter of any occurrence of the factor $1f$ in $\Delta L_C$, i.e. the integer $i$ satisfies \ $1f = \suff_{|1f|}(\pref_{i-3}(\Delta L_C))$.
    Let $S$ be a caterpillar subsequence of a caterpillar sequence of $C$ such that $\leaf{S} = L_C(i)$ and $\size{S}=i$.
    Also, let $A = \Left_{i-|1f|}(S)$ and $B = \Right_{|1f| + 3}(S)$ so that, by Lemma~\ref{l:graft}, we have $S = A \diamond B$.
    
    Then
    \begin{alignat*}{2}
    \leaf{B} & = \leaf{S} - \leaf{A}+2 & & \qquad \mbox{(by Equation~\eqref{EQ:n1})}\\
    & = L_C(i) - \leaf{A} + 2 & & \qquad \mbox{(by  hypothesis)}\\
    & \geq L_C(i) - L_C(i-|1f|) + 2
    && \qquad \mbox{(by definition of $L$)} \\
    & = |\suff_{|1f|}(\pref_{i-3}(\Delta L_C))|_1 +2
    && \qquad \mbox{(by Equation~\eqref{EQ:LDL})}  \\
    & = |1f|_1 + 2
    && \qquad \mbox{(by hypothesis)} \\
    & = |f|_1 + 3 \\
    & = |p|_1 + 3 \\
    & > |p|_1 + 2 \\
    & = |p0|_1 + 2 \\
    & = L_C(|p0|+3)
    && \qquad \mbox{(by Equation~\eqref{EQ:LDL} with $i=3$)} \\
    \end{alignat*}
    which is absurd since $\size{B} = |p0| + 3$ and $B\preceq S$.
\end{proof}

\section{Caterpillars and prefix normal forms} \label{sec:forms}

In~\cite{Burcsi}, Burcsi et al. introduced an equivalence relation on binary words as follows. Let $w$ and $w'$ be two binary words on $\{0,1\}$. We write $w \equiv w'$, if $F_1(w,i)=F_1(w',i)$ for all $i\in\N$,
where $F_1 : \{0,1\}^*\times \N \to \N$ is the function that associates to a binary word $u$ and an integer $i$  the maximal number of $1$'s occurring in any factor of length $i$ of $u$:
$$F_1(u,i) = \max \{|v|_1 : v\in \Fact_i(u)\}.$$
It was shown by the same authors that there exists a unique prefix normal word in each equivalence class~\cite[Theorem~2]{Burcsi}.

We now consider an equivalence relation on graphs whose restriction to caterpillar graphs is essentially the same as $\equiv$ on binary words.
Given two graphs $G$ and $G'$, we say that $G$ and $G'$ are \emph{leaf-equivalent} if $L_G = L_{G'}$.
As discussed throughout this paper, caterpillar graphs and caterpillar sequences are essentially the same objects.
Hence, we are allowed to restrict our attention on caterpillar sequences, keeping in mind that their properties have a twin counterpart in the set of caterpillar graphs.
Therefore, given two caterpillar sequences $S$ and $S'$, we say that $S$ and $S'$ are \emph{leaf-equivalent} if $L_S = L_{S'}$.

\begin{theorem}\label{thm:equiv}
    Consider two binary words $w$ and $w'$.
    We have
    $$L_{\RC(w)}=L_{\RC(w')}\text{ if and only if }F_1(w,i)=F_1(w',i) \text{ for all } i\in\N.$$
\end{theorem}

The proof of Theorem~\ref{thm:equiv} is based on the following lemmas and corollary.

\begin{lemma}\label{lem:word_to_tree}
    Let $w$ be a binary word and $u$ a factor of $w$. Then $\RC(u)$ is a subsequence of $\RC(w)$ with size $|u|+3$ and $|u|_1+2$ leaves.
\end{lemma}

\begin{proof}
    Let $w\in\{0,1\}^*$ and $u\in\Fact(w)$.
    Then there exist $x,y\in\{0,1\}^*$ such that $w=xuy$.
    By Lemma~\ref{L:cwt}, we have $\RC(w)=\RC(x)\diamond\RC(u)\diamond\RC(y)$. So $\RC(u)$ is a subsequence of $\RC(w)$.
    To conclude the proof note that $\size{\RC(u)}=|u|+3$ by Lemma~\ref{L:reading}~\eqref{L:size} and
    $\leaf{\RC(u)}=|u|_1+2$ by Lemma~\ref{L:reading}~\eqref{L:riv}.
\end{proof}

\begin{lemma}\label{lem:fully_is_left_right}
    Let $w$ be a binary word.
    For each integer $i\in\{3,\ldots,|w|+3\}$, there exist
    integers $j$ and $j'$ such that
    $$\Left_i(\Right_j(\RC(w))) \quad \text{and} \quad \Right_i(\Left_{j'}(\RC(w))),$$
    are fully leafed caterpillar subsequences.
\end{lemma}

\begin{proof}
    Let $n$ be the length of $\RC(w)$ and write $\RC(w)=(r_1,\cdots,r_n)$.
    Consider a fully leafed caterpillar subsequence $S=(s_1,\ldots,s_k)$
    of $\RC(w)$ of size $i$ such that
    $$S\neq \Left_i(\Right_j(\RC(w)))\text{ for all}\; j\ge i.$$
    As $S$ is a subsequence of $\RC(w)$, it occurs with a shift $\ell$, i.e.
    $$\spine(S)  [m]\le \spine(\RC(w))[\ell+m].$$
    
    Let $d=\sum_{m=2}^{k-1} r_{\ell+m} - s_m$.
    We show by contradiction that $s_1+s_k-2\ge d$.
    Indeed, assume that $s_1+s_k-2 < d$.
    Let $T$ be the subtree corresponding to $S$.
    The idea is to first consider the sequence $(s_2+1,\ldots,s_{k-1},1)$ that corresponds to $T$ where we removed $s_1+s_k-1\le d$ vertices.
    Then we can add $s_1+s_k-1\le d$ leaves to the previous tree.
    More formally, we choose an integer $z\in\{2,\ldots,k-1\}$ maximal such that $s_1+s_k-1\ge  \sum_{m=2}^z r_{\ell+m}-s_m$. Set $\xi=s_1+s_k-1 - \sum_{m=2}^z r_{\ell+m}-s_m$,
    so that the sequence
    $$S'=\left(r_{\ell+2}+1,r_{\ell+3},\ldots,r_{\ell+z}, s_{z+1}+\xi, s_{z+2},\ldots,s_{k-1},1\right)$$
    is a caterpillar subsequence of $\RC(w)$ and satisfies $\size{S'}=\size{S}$ and $\leaf{S'}=\leaf{S}+1$. This contradicts the fact that $S$ is a fully leafed subsequence.
    
    Hence, we have $s_1+s_k-2\ge d$.
    Therefore, $S'=(s'_1 ,r_{\ell+2},\ldots,r_{\ell+k-1},s'_k)$ with $s'_1=s_1-\min(s_1-1,d)$,  $s'_k=s_1+s_k-d-s'_1$ is a fully leafed subsequence (as the number of leaves did not change).
    Moreover, this sequence is of the form $\Left_i(\Right_j(\RC(w)))$ with
    $$j=n- \sum_{m =1}^{\ell} (r_m +1) - r_{\ell+1}+s'_1.$$
    Figure~\ref{fig:fully_is_left_right} illustrates the transformation $S\mapsto S'$. The red caterpillar subsequence $(3,0,2)$ is mapped to the blue one $(3,1,1)=\Left_8(\Right_{12}(RC(00110101100)))$.
    
    The case for $\Right_i(\Left_{j'}(\RC(w)))$ is symmetric.
\end{proof}

Observe that a fully leafed caterpillar subsequence of $\RC(w)$ is not always of the form $\Left_i(\Right_j(\RC(w)))$. For example, let $w=00110101100$ and consider the graph of $\RC(w)$ depicted in Figure~\ref{fig:fully_is_left_right}. We have $L_{\RC(w)}(8)=5$ and the subtree highlighted in red is fully leafed but cannot be obtained as $\Left_8(\Right_j(\RC(w)))$ for any $j\ge 8$.

\begin{figure}[ht]
    \begin{center}
        \includegraphics{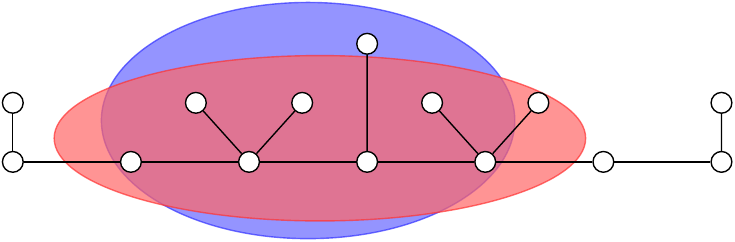}
    \end{center}
    \caption{Two fully leafed subtrees of $\RC(00110101100)$.}\label{fig:fully_is_left_right}
\end{figure}

\begin{corollary}\label{coro:fully_is_left_right}
    Let $w$ be a binary word. For each given size $i \geq 3$, there exists a factor $u$ of $w$ of length $i - 3$ such that $\RC(u)$ is fully leafed subsequence.
\end{corollary}

\begin{proof}
    Let $n=|w|$ and write $w=w_1\cdots w_n$ with $w_k\in\{0,1\}$. Recall that $\RC(w)$ has size $n+3$.
    Let $i\in\{4,\ldots,n+3\}$.
    By Lemma~\ref{lem:fully_is_left_right}, there exists an index $j\ge i$ such that   $\Left_i(\Right_j(\RC(w)))$ is a fully leafed subsequence of $\RC(w)$ with size $i$.
    Finally, let $u = w_{n-j+4}\cdots w_{n-j+i}$. Then by Lemma~\ref{L:reading}~\eqref{L:ri} and~\eqref{L:rii}, we have
    $$ \Left_i(\Right_j(\RC(w))) = \Left_i(\RC(\suff_{j-3}(w))) =  \RC(u)$$
    and $|u|= i-3$ as required.
\end{proof}

We now have the necessary tools to prove Theorem~\ref{thm:equiv}.

\begin{proof}[Proof of Theorem~\ref{thm:equiv}]
    Assume first that $F_1(w,i)> F_1(w',i) \text{ for some } i\in\N$.
    So there exists  $u$ a factor of $w$ of length $i$ such that $|u|_1>|u'|_1$ for any factor $u'$ of $w'$ of length $i$.
    Then by Lemma~\ref{lem:word_to_tree}, $\RC(u)$ is a subsequence of $\RC(w)$ with strictly more leaves than any subsequence of $\RC(w')$ that is equal to a $\RC(u')$ for some factor $u'$ of $w'$ of length $i$.
    We conclude that $L_{\RC(w)}(i)>L_{\RC(w')}(i)$ by Corollary~\ref{coro:fully_is_left_right}.
    
    Assume now that $L_{\RC(w)}\neq L_{\RC(w')}$. Let $n=|w|$ and $w=w_1w_2\cdots w_n$ with $w_j\in\{0,1\}$.
    Then, without loss of generality, there exists an integer $i$ such that $L_{\RC(w)}(i) > L_{\RC(w')}(i)$.
    Let $S_w$ be a fully leafed subsequence of $\RC(w)$ such that $\size{S_w}=i$ and $\leaf{S_w}=L_{RC(w)}(i)$. By Lemma~\ref{lem:fully_is_left_right}, we can suppose that 
    $S_w$ is such that 
    there exists an integer $k$ such that
    $$\RC(w) = \Left_k(\RC(w)) \diamond S_w \diamond \Right_{n+9-k-i} (\RC(w)).$$
    Then, by Equation~\eqref{EQ:n1} and Lemma~\ref{L:reading},
    \begin{alignat*}{2}
    & L_{\RC(w)}(i) =\leaf{S_w}\\
    & = \leaf{\RC(w)} - \leaf{\Left_k(\RC(w))} - \leaf{\Right_{n+9-k-i} (\RC(w))}+4\\
    & = |w|_1 +2 - |\pref_{k-3}(w)|_1 -2 - |\suff_{n+6-k-i}(w)|_1 - 2+4\\
    & = |w_{k-2}w_{k-1}\cdots w_{k+i-6}|_1 +2.
    \end{alignat*}
    So there is a factor $u=w_{k-2}w_{k-1}\cdots w_{k+i-6}$ of length $i-3$ of $w$ that contains $L_{\RC(w)}(i) -2$ times the letter $1$.
    Therefore
    \begin{alignat*}{2}
    |u|_1 & = L_{\RC(w)}(i) - 2 \\
    & > L_{\RC(w')}(i) - 2  &&\quad \mbox{(by hypothesis)}\\
    & \ge |u'|_1  && \quad \mbox{(by Lemma~\ref{lem:word_to_tree})}
    \end{alignat*}
    for any $u'$ factor of $w'$ such that $|u'|=i-3$. Hence $F_1(w,i-3)>F_1(w',i-3)$ and $w\not\equiv w'$.
\end{proof}


\section{Perspectives}\label{sec:concl}

Given a family $\mathcal{A}$ of graphs, let $\L(\mathcal{A})$ be the language of all possible leaf words $\Delta L_G$ for $G \in \mathcal{A}$. Let $\Graphs$ be the family of all graphs, $\Trees$ the family of all trees and $\Cater$ the family of all caterpillars.
Lemma~\ref{prop:alph-graph} implies $\L(\Trees) \cap \L(\Graphs \setminus \Trees) = \emptyset$.

As caterpillars are trees, one might wonder whether the language of  leaf words of caterpillars and the language of leaf words of trees are identical. This is not the case: Figure~\ref{fig:ternary_tree}  shows a tree  $T$ which is the smallest counter-example.   Indeed the leaf word  $w=1101011011$ of $T$  is not  prefix normal because
$$|\pref_5(w)|_1 = |11010|_1 = 3 < 4 = |11011|_1 = |\suff_5(w)|_1.$$
Hence,  $\L(\Cater) \subsetneq \L(\Trees)$.
The relation between the different classes of graphs considered in the paper is summarized in Figure~\ref{F:venn}.

\begin{figure}[ht]
    \centering
    \begin{tikzpicture}
    \draw (0,0) rectangle (6,3);
    \draw (2,0) -- (4,3);
    \draw (1.3,1.5) ellipse (1.2cm and 1cm);
    \node[above right] at (6,3)     {$\L(\Graphs)$};
    \node[below right] at (0,3)     {$\L(\Trees)$};
    \node[below right] at (0.8,2.3) {$\L(\Cater)$};
    \node[below left]  at (6,3)     {$\L(\Graphs \setminus \Trees)$};
    \end{tikzpicture}
    \caption{The relations between the languages $\L(\Graphs)$, $\L(\Trees)$, $\L(\Graphs \setminus \Trees)$ and $\L(\Cater)$.}\label{F:venn}
\end{figure}
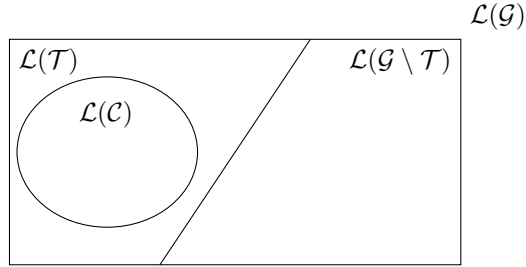

In order to investigate the language $\L(\Trees)$ of trees,
we have first extended the notion of prefix normal word to a notion of $k$-prefix normal word different from 
the one described in \cite{fici} as follows.
We say that a word $w$ is \emph{$k$-prefix normal} if for all $p \in \Pref(w)$ and for all $f \in \Fact_{|p|}(w)$ such that $|p| = |f|$, we have $|f|_1 - |p|_1 \leq k$.
One might wonder if there exists a constant $k$ such that $\L(\Trees) \subseteq \kPNW$, where $\kPNW$ is the set of all $k$-prefix normal words.
Unfortunately, this is not the case, since there exists a family $F_k$ of trees whose leaf words are $k$-prefix normal for every positive integers $k$ but not $(k-1)$-prefix normal.
Such a family $F_k$ is constructed as follows: Start with a single vertex, three simple chains of length $k-1$ and three star graphs with $k+3$ vertices as shown in Figure \ref{fig:k-prefix}.
Then add three edges connecting the isolated vertex to each simple chain and add three more edges connecting the other end of the simple chains to each star graph.
The leaf word associated with this graph is $1^{k+1}0^{k}10^{k}1^{k+1}0^{k}1^{k+1}$ and it is $k$-prefix normal but not $(k-1)$-prefix normal.
The graphs $F_1, F_2$ and $F_3$ are illustrated in Figure~\ref{fig:k-prefix}. It is easy to see that the leaf word of $F_2$ is $1^30^210^21^30^21^3$ which is $2$-prefix normal but not $1$-prefix normal.

\begin{figure}[ht]
	\begin{center} \begin{tabular}{ccc}
	\begin{minipage}[c]{.3\linewidth}
		\includegraphics[scale=1]{ternary_tree}
	\end{minipage}
  &
	\begin{minipage}[c]{.3\linewidth}
		\includegraphics[scale=1]{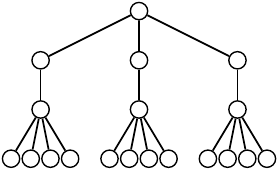}
	\end{minipage}
  &
	\begin{minipage}[c]{.3\linewidth}
		\includegraphics[scale=1]{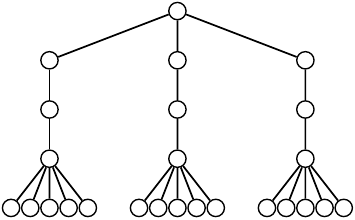}
	\end{minipage}
  \\ [1cm]
  $F_1$
  &
  $F_2$
  &
  $F_3$
	\end{tabular} \end{center}
	\caption{The trees $F_1$, $F_2$ and $F_3$.}
	\label{fig:k-prefix}
\end{figure}

\section*{References}\label{sec:biblio}
\bibliography{words-fully-leafed}

\newpage

\appendix
\section{Proofs}\label{sec:appendix}

We now give the details of the proofs omitted in Section~\ref{sec:caterpillar_sequences}.

\begin{proof}[Proof of Proposition~\ref{P:poset}]
    
    \emph{Reflexivity}. Immediate.
    
    \emph{Antisymmetry}. Let $S=(s_1,s_2,\dots, s_{k_S})$ and $T=(t_1,t_2,\dots, t_{k_T})$ be two caterpillar sequences such that $S\preceq T$ and $T\preceq S$.
    Then we have $k_S = k_T$ so that the shift is $i = 0$.
    Since $S\preceq T$, we have $\spine(S)[j] \leq \spine(T)[j]$ for all $j \in \{ 1,2,\ldots,k_S\}$.
    Similarly, since $T \preceq S$, $\spine(S)[j] \geq \spine(T)[j]$.
    This implies that $S = T$.
    
    \emph{Transitivity}. Consider three caterpillar sequences, $S=(s_1,s_2,\dots, s_{k_S})$, $T=(t_1,t_2,\dots, t_{k_T})$ and $U=(u_1,u_2,\dots, u_{k_U})$, such that $S\preceq T$ and $T\preceq U$.
    By definition, there exist two shifts $i \in \{0,1,\dots, k_T-k_S\}$ and $\ell \in \{0,1,\dots, k_U-k_T\}$ such that
    \begin{alignat*}{2}
    \spine(S)[j] & \leq \spine(T)[j+i],    \quad && \mbox{for $j=1,2,...,k_S$;} \\
    \spine(T)[j] & \leq \spine(U)[j+\ell], \quad && \mbox{for $j=1,2,...,k_T$.}
    \end{alignat*}
    Let $j \in \{1,2,\dots,k_S\}$. Then $1 \leq j + i \leq k_S + (k_T - k_S) = k_T$, which implies
    $$\spine(S)[j] \leq \spine(T)[j+i] \leq \spine(U)[j+i+\ell].$$
    Since $0\leq i + \ell \leq k_U - k_S$, we conclude that $S \preceq U$.
\end{proof}

\begin{proof}[Proof of Lemma~\ref{L:subs}]
    (i) We proceed by induction on $n = |S|$.
    
    \textsc{Basis}. If $n = 3$, then $S = (2)$ and $i = 3$. Thus
    $\Left_i(S) = (2) = S \preceq S$.
    
    \textsc{Induction.} Assume that $\Left_i(S') \preceq S'$ for any caterpillar sequence $S'$ of size $n' < n$ and let $S = (s_1,\ldots,s_k)$ be of size $n$. If $i = n$, then
    $\Left_i(S) = S \preceq S.$ Suppose $i\in\{3,\ldots,n-1\}$.
    
    On one hand, if $s_k \geq 2$, then
    \begin{alignat*}{2}
    \Left_i(S)
    & = \Left_i(s_1,\ldots,s_{k-1},s_k - 1)
    && \qquad \mbox{(by Equation~\eqref{EQ:left})}\\
    & \preceq (s_1,\ldots,s_{k-1},s_k - 1)
    && \qquad \mbox{(by induction hypothesis)}\\
    & \preceq (s_1,\ldots,s_{k-1},s_k)
    && \qquad \mbox{(by Definition~\ref{D:subsequence})}\\
    & = S.
    \end{alignat*}
    
    On the other hand, if $s_k = 1$, then
    \begin{alignat*}{2}
    \Left_i(S)
    & = \Left_i(s_1,\ldots,s_{k-1} + 1)
    && \qquad \mbox{(by Equation~\eqref{EQ:left})}\\
    & \preceq (s_1,\ldots,s_{k-1} + 1)
    && \qquad \mbox{(by induction hypothesis)}\\
    & \preceq (s_1,\ldots,s_{k-1},s_k)
    && \qquad \mbox{(by Definition~\ref{D:subsequence})}\\
    & = S,
    \end{alignat*}
    since $\spine(s_1,\ldots,s_{k-1} + 1)[k-1]\leq(s_{k-1} + 1)+1 = \spine(S)[k-1]$.
    
    (ii) Symmetric to (i).
    
    (iii) Follows from the symmetry of Equations~\eqref{EQ:left} and \eqref{EQ:right}.
\end{proof}

\begin{proof}[Proof of Lemma \ref{L:alpha-beta}]
We only prove the statement about $\Left_i(S)$, since the proof for  $\Right_i(S)$ is symmetric.
The proof is done by induction on $n = |S|$, with $i$ fixed.

\textsc{Basis}. If $n = i$, then $\Left_i(S) = \Left_{\size{S}}(S) = S$, which implies $a = k - 1$ and $\alpha = s_k$ so that  Relations~\eqref{EQ:alpha} are true.

\textsc{Induction.} Assume that the result holds for $n - 1$. Since $n > i$, by Equation~\ref{EQ:left}, there exists a caterpillar sequence $S' = (s'_1,\ldots,s'_{k'})$ of size $n - 1$, such that $\Left_i(S) = \Left_i(S')$, where
$$S' = \begin{cases}
  (s_1,\ldots,s_{k-1},s_k - 1), & \mbox{if $s_k \geq 2$;} \\
  (s_1,\ldots,s_{k-1} + 1),     & \mbox{if $s_k = 1$.}
\end{cases}$$
Therefore, by the induction hypothesis, there exist unique integers $a$ and $\alpha$ such that $\Left_i(S') = (s'_1,\ldots,s'_{a},\alpha)$ with
$$0 \leq a \leq k' - 1, \quad 1 \leq \alpha \leq s'_{a+1} + 1 \quad \text{and} \quad i = \sum_{m=1}^{a} (s'_m + 1) + (\alpha + 1).$$

We claim that 
\begin{equation*}\label{EQ:left-left}
\Left_i(S) = \Left_i(S') = (s'_1,\ldots,s'_{a},\alpha) = (s_1,\ldots,s_{a},\alpha)
\end{equation*}
To prove this claim, note that 
if $s_k \geq 2$, then $(s'_1,\ldots,s'_{k'}) = (s_1,\ldots,s_{k-1},s_k-1)$ and $k' = k$. Since $a \leq k' - 1 = k - 1$, we have $s'_m = s_m$ for $m = 1,2,\ldots,a$. 
Similarly, if $s_k = 1$, then $(s'_1,\ldots,s'_{k'}) = (s_1,\ldots,s_{k-1}+1)$ and $k' = k - 1$. Since $a \leq k' - 1 = k - 2$, we also have $s'_m = s_m$ for $m = 1,2,\ldots,a$, 

It remains to show that the integers $a$ and $\alpha$ satisfy Relations~\eqref{EQ:alpha}.
On one hand, if $s_k\ge2$, we have
$$0 \leq a \leq k' - 1 = k - 1,$$
$$1 \leq \alpha \leq s'_{a+1} + 1 \leq s_{a+1} + 1$$
and, by the induction hypothesis,
$$i = \sum_{m=1}^{a} (s'_m + 1) + (\alpha + 1) = \sum_{m=1}^a (s_m + 1) + (\alpha + 1).$$
On the other hand, if $s_k=1$, we obtain
$$0 \leq a \leq k' - 1 = k - 2 \leq k - 1,$$
$$1 \leq \alpha \leq s'_{a+1} + 1 \leq s_{a+1} + 1$$
and, by the induction hypothesis,
$$i = \sum_{m=1}^{a} (s'_m + 1) + (\alpha + 1) = \sum_{m=1}^a (s_m + 1) + (\alpha + 1).$$
\end{proof}

\begin{proof}[Proof of Lemma~\ref{L:cwt}]
    (i) Let $u$ and $v$ be binary words. We proceed by induction on the length of $v$.
    
    \textsc{Basis.} For $v=\varepsilon$, we obtain $\RC(uv)=\RC(u)=\RC(u)\diamond (2) =\RC(u)\diamond \RC(\varepsilon)$ since $(2)$ is the identity of the graft.
    
    \textsc{Induction.} Let $n\ge 1$. Assume that $\RC(u)\diamond \RC(v)=\RC(uv)$
    for any binary word $v$ such that $|v|<n$. Consider $|v|=n$ and  write $v=v'a$ where $v'\in\{0,1\}^*$ and $a\in\{0,1\}$.
    We set 
    $$S=\begin{cases}
         (1,1), & \text{if }a=0;\\
         (3),   &\text{if }a=1.\\
        \end{cases}$$    
    We have
    \begin{alignat*}{2}
    \RC(uv)&=\RC(uv'a) = \RC(uv')\diamond S &&\quad \mbox{(by Definition~\ref{D:reading})}\\
    & = (\RC(u)\diamond\RC(v'))\diamond S &&\quad \mbox{(by induction hypothesis)}\\
    & = \RC(u)\diamond(\RC(v')\diamond S) &&\quad \mbox{(by associativity of $\diamond$)}\\
    & = \RC(u)\diamond \RC(v'a) &&\quad  \mbox{(by definition of $\diamond$)}\\
    & = \RC(u)\diamond \RC(v).&&
    \end{alignat*}
    
    (ii) We proceed by induction on the length of the binary word $w$.
    
    \textsc{Basis.} For $w=\varepsilon$, we clearly have $\RC(\tilde{\varepsilon})=(2)=\widetilde{\RC(\varepsilon)}$.
    
    \textsc{Induction.} Assume that $\RC(\tilde{v})=\widetilde{\RC(v)}$ for any binary word $v$ of length $|v|<|w|$.
    Let $w=ua$ where $u\in\{0,1\}^*$ and $a\in\{0,1\}$.
    Then
    \begin{alignat*}{2}
    \RC(\tilde{w}) & =\RC(\widetilde{ua}) = \RC(a\tilde{u})&&\\
    &  = \RC(a)\diamond \RC(\tilde{u}) && \quad \mbox{(by the previous point)}\\
    & =  \widetilde{\RC(a)}\diamond \widetilde{\RC(u)} &&\quad  \mbox{(by induction hypothesis)}\\
    & = \widetilde{\RC(u)\diamond\RC(a)} && \quad \mbox{(by Lemma \ref{L:cwt})}\\
    & = \widetilde{\RC(ua)}=\widetilde{\RC(w)}&& \quad \mbox{(by definition of $w$)}
    \end{alignat*}
    as required.
\end{proof}
\begin{proof}[Proof of Lemma~\ref{L:reading}]
    Let $\RC(w) = (r_1,\dots,r_k)$.
    
    (i) 
    We proceed by induction on $|w|$.
    
    \textsc{Basis.} If $|w| = 0$, then $i = 3$ and we have
    $$\Left_3(\RC(w)) = (2) = \RC(\varepsilon) = \RC(\pref_{3-3}(w)) = \RC(\pref_{i-3}(w)).$$
    
    \textsc{Induction.} Assume that $w = ub$ for some word $u$ and letter $b$. If $3 \le i < |w| + 3$, then $0 \le i - 3 < |w| = |u| + 1$ and $i - 3 \leq |u|$. Therefore,
    \begin{alignat*}{2}
    \RC&(\pref_{i-3}(w))\\
    & = \RC(\pref_{i-3}(ub)) \\
    & = \RC(\pref_{i-3}(u))
    && \qquad \mbox{(since $i - 3 \leq |u|$)} \\
    & = \Left_i(\RC(u))
    && \qquad \mbox{(by induction hypothesis)} \\
    & = \begin{cases}
    \Left_i(r_1,\ldots,r_{k-1} + 1), & \mbox{if $r_k = 1$;} \\
    \Left_i(r_1,\ldots,r_{k-1}, r_k - 1), & \mbox{if $r_k \geq 2$;} \\
    \end{cases}
    && \qquad \mbox{(by Equation~\ref{EQ:reading})} \\
    & = \Left_i(\RC(w))
    && \qquad \mbox{(by Equation~\ref{EQ:left})}
    \end{alignat*}
    as required. 
    
    Finally, if $i = |w| + 3$, we have
    \begin{align*}
    \Left_{|w|+3}(\RC(w)) = \RC(w)
    = \RC(\pref_{|w|}(w)) =  \RC(\pref_{i-3}(w))
    \end{align*}
    concluding the proof.
    
    (ii) It follows from (i), Lemmas~\ref{L:cwt}~\eqref{L:cwtii} and~\ref{L:subs}~\eqref{L:subsiii} that
    \begin{alignat*}{2}
    \Right_i(\RC(w))
    & = \left(\Left_i(\widetilde{\RC(w)})\right){\widetilde{\phantom{ab}}} \\
    & = \left(\Left_i(\RC(\widetilde{w}))\right){\widetilde{\phantom{ab}}} \\
    & = \left(\RC(\pref_{i-3}(\widetilde{w}))\right){\widetilde{\phantom{ab}}} \\
    & = \left(\RC(\widetilde{\suff_{i-3}(w)})\right){\widetilde{\phantom{ab}}} \\
    & = \RC(\suff_{i-3}(w)).
    \end{alignat*}
    
    (iii) By induction on $|w|$.
    
    \textsc{Basis.} If $|w|=0$ then $w=\varepsilon$ and
    $$\size{RC(w)}=\size{(2)} = 3 = |w|+3$$
    
    \textsc{Induction.} If $w=ua$ for some word $u$ and letter $a$. Let $RC(u)=(r_1,...,r_k)$, then
    \begin{align*}
        \size{RC(ua)}&=
        \begin{cases}
        \size{(r_1,...,r_k-1,1)}, &\mbox{if $a=0$}\\
        \size{(r_1,...,r_k+1)}, &\mbox{if $a=1$} 
        \end{cases} \\
        & =\size{(r_1,...,r_k)}+1\\
        &=|u|+3+1 \\
        &=|w|+3
    \end{align*}
    
    (iv) By Definition~\ref{D:reading}, on one hand, if $a = 0$, then
    $$\leaf{\RC(wa)} = \left(\sum_{j=1}^{k-1} r_i\right) + (r_k - 1) + 1  = \leaf{\RC(w)} = \leaf{\RC(w)} + a.$$
    On the other hand, if $a = 1$, then
    $$\leaf{\RC(wa)} = \left(\sum_{j=1}^{k-1} r_i\right) + (r_k+1)  = \leaf{\RC(w)} + 1 = \leaf{\RC(w)} + a.$$
    (v) We proceed by induction on $|w|$.
    
    \textsc{Basis.} If $|w| = 0$, then $w = \varepsilon$ and
    $$\leaf{\RC(w)} = \leaf{\RC(\varepsilon)} = \leaf{(2)} = 2 = |\varepsilon|_1 + 2 = |w|_1 + 2.$$
    
    \textsc{Induction.} Assume that $w = ub$ for some word $u$ and some letter $b$. Then
    $$\leaf{\RC(w)} = \leaf{\RC(ub)} = \leaf{\RC(u)} + b = (|u|_1 + 2) + b = |ub|_1 + 2 = |w|_1 + 2,$$
    where the second equality follows from~\eqref{L:riii} and the third equality holds by the induction hypothesis.
    
    (vi) We have
    $$\leaf{\Left_i(\RC(w))} = \leaf{\RC(\pref_{i-3}(w))} = |\pref_{i-3}(w)|_1 + 2,$$
    where the first equality follows from~\eqref{L:ri} and the second from~\eqref{L:riv}.
    
    (vii) Symmetric to~\eqref{L:rV}.
\end{proof}

\end{document}